\documentclass[reqno, draft]{amsart}
\usepackage{amssymb, amscd, amsfonts}
\usepackage[mathscr]{eucal}
\usepackage{amscd}

\newtheorem{theorem}{Theorem}[section]
\newtheorem{lemma}[theorem]{Lemma}
\newtheorem{corollary}[theorem]{Corollary}

\theoremstyle{definition}
\newtheorem{definition}[theorem]{Definition}
\newtheorem{example}[theorem]{Example}

\usepackage{a4wide}

\DeclareMathOperator{\McN}{\mathscr M}

\DeclareMathOperator{\conv}{{\rm conv}}
\DeclareMathOperator{\den}{{\rm den}}
\DeclareMathOperator{\ver}{{\rm vert}}
\DeclareMathOperator{\Zed}{\mathbb{Z}}
\newcommand{\R}{\mathbb{R}}
\newcommand{\Q}{\mathbb{Q}}
\newcommand{\cube}{[0,1]^{\it n}}

\title[Projective unital $\ell$-groups]{Rational Simplicial geometry and projective unital
lattice-ordered abelian groups}

\author[L.M.Cabrer]{Leonardo Manuel Cabrer}
\address{Department
 of Computer Science, Statistics and Applications ``Giuseppe Parenti'',
University of Florence,
Viale Morgagni 59 --
I-50134, Florence -
Italy}
\email{l.cabrer@disia.unifi.it}

\thanks{This research was supported by a Marie Curie Intra European Fellowship within the 7th European Community Framework Program (ref. 299401-FP7-PEOPLE-2011-IEF)}



\date{\today}
\begin{document}

\begin{abstract}
A {\it unital $\ell$-group} is an abelian group equipped
 with a translation invariant lattice-order 
and with a distinguished {\it strong unit}, 
i.e. an element whose positive
integer multiples eventually dominate every element of $G$.
If $X$ is a compact subset  of $\R^n$, 
the set $\McN(X)$ of real-valued piecewise linear maps with integer 
coefficients, 
whose addition and lattice operations defined pointwise and whose distinguished element is the constant map $1$, 
is a  unital $\ell$-group.

In this paper we provide a geometric decription of finitely generated projective unital $\ell$-groups. 
We prove that a finitely unital $\ell$-group is  projective 
if and only if it is isomorphic to $\McN(P)$ for some polyhedron $P$ which is rational,
contractible, contains an integer point, and satisfies an elementary
arithmetical-topological property.
\end{abstract}

\maketitle

\section{Introduction}

An {\it $\ell$-group} $G$ is an abelian group equipped with
a 
translation invariant lattice-order.
A (unital) $\ell$-group $(G,u)$ is {\it (regular) projective} 
if whenever $\psi\colon (A,a)\to(B,b)$ 
is a surjective (unital) $\ell$-homomorphism 
and $\phi\colon (G,u)\to(B,b)$ is a (unital) $\ell$-homomorphism 
(i.e. preserve the lattice and group structure and the unit), there
is a unital $\ell$-homomorphism $\theta\colon (G,u)\to(A,a)$ 
such that $\phi= \psi \circ \theta$.
Baker \cite{Bak1968} and Beynon \cite{Bey1977a}
proved that an $\ell$-group $G$
 is finitely generated projective 
if and only  it is finitely presented.
The characterisation of projective {\it unital} $\ell$-groups 
is far more delicate.

Given a rational polyhedron $P\subseteq[0,1]^n$ (see Section~\ref{Sec:RatPol}), 
we let \ $\McN(P)$ denote the set of all continuous functions
$f\colon P\to \R$ 
having the following property: there are finitely many linear
(in the affine sense) polynomials $p_{1},\ldots,p_{m}$ with integer coefficients, 
such that
for all $x\in [0,1]^{n}$ there is $i\in \{1,\ldots,m\}$ with
$f(x)=p_{i}(x)$.
The set $\McN(P)$ is the universe of 
a subalgebra of the unital $\ell$-group  $(\R,1)^{P}$.
Therefore,  $\McN(P)$ carries a natural structure of 
unital $\ell$-group, 
where the lattice operations and the group operations 
are defined pointwise from 
$\R$ and the strong unit is the constant map 
$v\in P\mapsto1\in\R$.

In \cite{CM2013} the author and D. Mundici presented a 
characterisation of projective unital $\ell$-groups 
as those isomorphic to  $\McN(P)$  
where the rational polyhedron $P\subseteq \cube$ 
is a special kind of retract of $\cube$ (for some $n=1,2,\ldots$)
called $\Zed$-retract. Since $\Zed$-retract are continuous 
retractions of cubes, they are trivially {\it contractible}, that is,  
homotopically equivalent to a point
 (see for example \cite[Chapter 0]{Hat2001})).
Another important property of $\Zed$-retracts is that they are 
{\it strongly regular} (see \cite[Definition~3.1]{CM2013} and subsection~\ref{SubSec:RegularTriang}). 
In \cite[Theorem~4.17]{Cab201X}, we observe how
strong regularity is connected with the notion of anchored 
polytopes defined by  Je\v{r}\'abek in \cite{Je2010}.
In Lemma~\ref{Lem:SRandConvexSet} we prove that a rational polyhedron $P$ is strongly regular if and only if for each $v\in P\cap \Q^n$ there exist $w\in \Zed^n$ and $\varepsilon>0$ such that the convex segment $\conv(v,v+\varepsilon (w-v))$ is contained in~$P$.
Using this result an equivalent presentation of \cite[Theorem~4.5]{CM2013} is as follows.

\begin{theorem}
\label{Thm:ToInvert}
If the unital $\ell$-group  $(G,u)$ 
is finitely generated and projective, 
then there exist $n\in\{1,2,\ldots\}$ 
and a rational  polyhedron  $P\subseteq \cube$ 
such that $(G,u)\cong\McN(P)$ and $P$ 
satisfies the following conditions:
\begin{itemize}
\item[(i)]
$P$ is contractible,

\item[(ii)] $P\cap\{0,1\}^n\neq\emptyset$, and

\item[(iii)]
for each $v\in P\cap \Q^n$ there exist $w\in \Zed^n$ and 
$\varepsilon>0$ such that the convex segment 
$\conv(v,v+\varepsilon (w-v))$ is contained in~$P$.
\end{itemize}
  \end{theorem}

This paper is devoted to proving the converse of 
Theorem~\ref{Thm:ToInvert}. 
In the special case when the rational polyhedron $P$
 is a finite union of $1$-simplexes, 
the converse was already proved in \cite[Corollary~5.4]{CM2013}, 
but the general case has remained open until now. 

In \cite{Mu1986}, Mundici proved that unital $\ell$-groups are 
categorically equivalent to the variety of MV-algebras 
(the algebraic semantic of \L ukasiewicz infinite valued logic). 
Since (regular) projective objects are preserved under categorical 
equivalences, the main result in this paper provides also a
 complete description of projective MV-algebras. 
Moreover, this paper solves the sixth problem in the list of open problems
presented by Mundici in \cite{Mu2011}.

\section{Simplicial Geometry}
In this section we recall some definitions and results 
about  simplicial complexes used in the rest of the paper. 
We refer the reader to~\cite{Ewa1996}  and~\cite{Sta1967} 
for background 
in elementary polyhedral topology and simplicial complexes.

\subsection{Triangulations and subdivisions.}

During the rest of the paper, unless  otherwise stated, 
 $m$ and $n$ will denote positive integers, 
that is,  $m,n\in\{1,2,\ldots\}$.

For each $0\leq k\leq n$, a {\it $k$-simplex} $S\subseteq\R^n$ is the convex hull 
$S =\conv(v_{0},\dots,v_k)$ of
affinely independent points $v_{0},\dots,v_k\in\R^n$  . 
We let $\ver(S)=\{v_{0},\dots,v_k\}$ 
denote the set of vertices of $S$.
For any $F\subseteq \ver(S)$ (including the empty set), 
the convex hull $\conv(F)$ is called a
{\it face} of $S$.

A {\it polyhedron} $P$  in $\R^{n}$  is a finite union of 
(always closed) simplexes $P=S_{1}\cup\cdots\cup S_{t}$ in $\R^n$.
Every simplicial complex considered in this paper 
will have only finitely many simplexes.
For any  simplicial complex $\Delta$, 
its {\it support} $|\Delta|$ is the point-set union 
of all simplexes of $\Delta$. 
The set of its vertices is denoted by $\ver(\Delta)$, that is, 
 $\ver(\Delta)$ is the set of the vertices of its simplexes 
$\ver(\Delta)=\bigcup\{\ver(S)\mid S\in\Delta\}$.
Given a polyhedron $P$,  
a {\it triangulation}  of $P$ is a simplicial complex $\Delta$
 such that $P=|\Delta|$.

Given a polyhedron $P$ and 
triangulations $\Delta$ and $\Sigma$ of $P$
we say that $\Delta$ is a {\it subdivision} of $\Sigma$ if every
simplex of $\Delta$ is contained in a simplex of $\Sigma$.
For any  point $p \in P$, let $\Delta_{(p)}$ 
be the subdivision of $\Delta$
given by replacing every simplex $S\in \Delta$ that contains $p$
by the set of all simplexes of the form $\conv(F\cup\{p\})$,
where $F$ is any (including the empty) face of $S$ 
that does not contain $p$. 
The triangulation $\Delta_{(p)}$  of $P$ is called the 
{\it elementary stellar subdivision $\Delta_{(p)}$ 
of $\Delta$ at~$p$} 
\cite[ Definition III.2.1]{Ewa1996}. 
Given two triangulations $\Delta$ and $\nabla$ of $P$, 
$\nabla$ is said to be a {\it stellar subdivision} of~$\Delta$
if there exists a finite sequence $p_1,\ldots,p_m\in P$ such that, 
writing $\Delta_0=\Delta$
and $\Delta_{i}=(\Delta_{i-1})_{(p_i)}$ for each $i= 1,\ldots,m$, 
it follows that $\nabla=\Delta_m$.

\begin{lemma}[{\cite[Theorem~1]{Whi1935}}]\label{Lem:BlowUpSubdivision}
Let $P\subseteq\R^n$ be a polyhedron 
and $\Delta$ and $\Sigma$ be triangulations of $P$.
There exists a stellar subdivision $\nabla$ of $\Delta$ 
that is a subdivision of $\Sigma$. 
Moreover, if $S\in \Delta$ is such that $S\subseteq T$ 
for some $T\in\Sigma$, then $S\in\nabla$.
\end{lemma}

\begin{corollary}
\label{Cor:Triang-Subset1}
Let $P\subseteq Q\subseteq\R^{n}$ be polyhedra 
and $\Delta$ a triangulation of $Q$.
Then there exists a stellar subdivision $\nabla$ of $\Delta$   
such that 
\begin{itemize}
\item[(i)] the simplicial complex $\nabla_{P}=\{S\in\nabla\mid S\subseteq P\}$ 
is a triangulation of $P$, and 
\item[(ii)] if $S\in\Delta$ is such that $S\subseteq P$, 
then $S\in\nabla$.
\end{itemize}
\end{corollary}

Given  polyhedra $P\subseteq\R^n$  and $Q\subseteq\R^m$,
 a map $\eta\colon P \to Q$ is said to be 
{\it (continuous) piecewise-linear} 
if there exists a triangulation $\Delta$  of $P$ such that 
 for each $T\in \Delta$ 
the restriction $\eta{\upharpoonright}_T$ of $\eta$ to $T$ 
is an affine linear map. 
In this case we say that $\eta$ and $\Delta$ are {\it compatible}.

\begin{corollary}\label{Cor:RefineByBlowUps}
Let $P\subseteq\R^n$  and $Q\subseteq\R^m$ be polyhedra, 
and let $\Delta$ and $\nabla$ be triangulations 
of $P$ and $Q$, respectively. 
If $\eta\colon P \to Q$ is 
a piecewise linear map compatible with $\Delta$, 
there exists a stellar subdivision $\Delta'$ of $\Delta$ such that 
\begin{itemize}
\item[(i)]for each $S\in\Delta'$, there exists $T\in\nabla$ with 
$\eta(S)\subseteq T$, and
\item[(ii)] if $S\in\Delta$ is such that there exists $T\in\nabla$ 
with $\eta(S)\subseteq T$, then $S\in\Delta'$.
\end{itemize}
\end{corollary}

\begin{proof}
Let $T_1,\ldots,T_s$ be the maximal simplexes in $\nabla$. 
Then $\eta^{-1}(T_i)$ is a polyhedron  for each $i=1,\ldots,s$. 
Repeated applications of Corollary~\ref{Cor:Triang-Subset1}, 
provide a triangulation $\Delta'$ of $P$ 
with the following properties:
\begin{itemize}
  \item[(i)] $\Delta'$ is a stellar subdivision of $\Delta$, and
  \item[(ii)] for each $i=1,\ldots,s$ the simplicial complex
 $\Delta'_{i}=\{S\in\Delta'\mid S\subseteq \eta^{-1}(T_i)\}$ 
is a triangulation of $\eta^{-1}(T_i)$.
\end{itemize}
Since 
$P=\eta^{-1}(Q)=\eta^{-1}(T_1)\cup\cdots\cup \eta^{-1}(T_s)$,
condition (ii) is to the effect that each $S\in \Delta'$ 
is contained in $\eta^{-1}(T_i)$ for some $i=1,\ldots,s$. 
Therefore, $S\subseteq \eta^{-1}(T_i)$, that is, 
$\eta(S)\subseteq T_i$. 
\end{proof}

\subsection{Abstract simplicial complexes and collapses}\label{Subsec:AbsSimCom}

Let us recall that 
a {\it {\rm (}finite{\rm)} abstract simplicial complex} 
is a pair $({V},\Sigma)$, where ${V}$ is a finite set, 
whose elements are called the {\it vertices} of $({V},\Sigma)$,
and $\Sigma$ is a collection of subsets of $V$
whose union is ${V}$,
having the property that every subset of an element of $\Sigma$ is
again an element of $\Sigma$.

Two abstract simplicial complexes $(V,\Sigma)$ and $( V',\Sigma')$
are said to be {\it  isomorphic} 
if there is a one-one map $\gamma$ from $V$ onto $V'$
such that  
\[
  \{w_{1},\ldots,w_{k}\}\in \Sigma
        \mbox{ if and only if }
  \{\gamma(w_{1}),\ldots,\gamma(w_{k})\}\in \Sigma'
\]
for each   $\{w_{1},\ldots,w_{k}\}\subseteq  V$.

For every complex $\Delta$,  the {\it skeleton} of  $\Delta$
is the  abstract simplicial complex $(\ver(\Delta),\Sigma_{\Delta})$, 
where
\[
  \Sigma_{\Delta}=\{\{w_{1},\ldots,w_{k}\}\subseteq 
  \ver(\Delta)\mid \conv(w_{1},\ldots,w_{k})\in\Delta\}.
\]
Two simplicial complexes $\Delta$ and $\nabla$ are 
 {\it  simplicially isomorphic} if their skeletons are isomorphic.

 A face $F$ of a simplex $S$ is called {\it a facet of $S$} 
if it is maximal among the proper faces of $S$.
A simplex $T$ of a simplicial complex $\nabla$ in $\cube$ 
has a {\it free face} $F$ (relative to $\nabla$) 
if $F$ is a  facet of $T$, but $F$ is a face 
of no other simplex of~$\nabla$.
It follows that $T$ is a maximal simplex of $\nabla$,
and the removal  of both $T$ and~$F$ from $\nabla$
results in the subcomplex 
$\nabla'= \nabla\setminus\{T,F\}$ of $\nabla$.
The transition  from $\nabla$ to $\nabla'$ 
is an {\it elementary collapse}  \cite[Definition~III.7.2]{Ewa1996}. 
(In \cite[p.247]{Whi1939}, elementary collapses 
are called ``elementary contractions''.)
If a simplicial complex $\Delta$ can be obtained from $\nabla$
by a sequence of elementary collapses
we say that $\nabla$  {\it collapses to}  $\Delta$.
We say that $\nabla$ is {\it collapsible} if it collapses to 
(the simplicial complex consisting of) 
one of its vertices.

\begin{example}\label{Ex:CollapsibleCube}
For each $n=1,2,\ldots$, $k=0,1,2,\ldots$, and points $v_0,\ldots, v_k\in\R^n$ 
the convex hull $\conv(v_0,\ldots, v_k)$ 
admits a collapsible triangulation $\Delta$. 
Further, $\Delta$ can be obtained in such a way that 
$\ver(\Delta)\subseteq\{v_0,\ldots, v_k\}$ 
(see for example \cite[Theorem~2.6]{Ewa1996}). 

In particular, the cube $\cube$ admits a collapsible triangulation.
Indeed, 
an example of a collapsible triangulation of the cube $\cube$
 is given by the  {\em standard triangulation} \cite[p. 60]{Se1982}, 
formed by the simplexes $\conv(C)$ 
whenever $C$ is a chain in  $\{0,1\}^{n}$ with the product order 

\end{example}

The following result is a direct application of 
\cite[Theorem~4]{Whi1939}.

\begin{lemma}\label{Lem:BlowUpsPreserveCollapsible}
Let $P\subseteq\R^n$ be a polyhedron 
and $\Delta$ a triangulation of $P$. 
If $\Delta$ is collapsible, 
then every stellar subdivision $\nabla$  of $\Delta$ is collapsible.
\end{lemma}

\section{Rational polyhedra and $\Zed$-maps}

\subsection{Rational polyhedra and triangulations}\label{Sec:RatPol}

A simplex $S$  is said to be {\it rational} 
if the coordinates of each $v\in\ver(S)$ are rational numbers.  
A set $P\subseteq \R^n$ is said to be a {\it rational polyhedron} 
if there are rational simplexes $T_{1},\ldots,T_{l}\subseteq\R^n$
such that $P=T_{1}\cup\cdots\cup T_{l}$. 

We say that a simplicial complex  $\Delta$ is {\it rational} 
if all simplexes of $\Delta$ are rational.
Given a rational polyhedron $P$,   a {\it rational triangulation}  
of $P$ is a triangulation of $P$ that is a rational simplicial complex.

\begin{theorem}[{\cite[Theorem 1]{Bey1977}}]
Let $P\subseteq \R^n$ be a rational polyhedron, 
and let $\Delta$ be a triangulation of $P$ with vertices 
$v_1,\ldots,v_k$. 
For any $0<\varepsilon \in\R$, 
there is a rational triangulation $\nabla$ of $P$
with vertices $w_1,\ldots,w_k$ such that:
\begin{itemize}
  \item[(i)] $||v_i -w_i|| < \varepsilon$ for each  $i = 1,\dots, k$, and
  \item[(ii)] the map sending the vertex $v_i$  of $\Delta$ 
     to the vertex $w_i$ of $\nabla$ defines a simplicial isomorphism 
     between $\Delta$ and $\nabla$.
\end{itemize}
\end{theorem}

\begin{corollary}\label{cor:Beynon} 
Let $P\subseteq Q\subseteq\R^n$ be such that 
$P$ is a rational polyhedron and $Q$ a convex polyhedron. 
Let $\Delta$ be a triangulation of $Q$ 
such that the simplicial complex 
$\Delta_{P}=\{S\in\Delta\mid S\subseteq P\}$ 
is a triangulation of $P$. 
Then there is triangulation $\nabla$ of~$Q$ such that:
\begin{itemize}
  \item[(i)] the simplicial comlex $\nabla_P=\{S\in\nabla\mid S\subseteq P\}$ 
     is a rational triangulation of~$P$, and
  \item[(ii)] $\nabla$ and $\Delta$ are simplicially isomorphic.
\end{itemize}
\end{corollary}

\begin{lemma}\label{Lem:RationalizationOfMap}
Let $P\subseteq\R^n$  and $Q\subseteq \R^m$ 
be rational polyhedra, 
$\Delta$ a {\rm (}not necessarily rational{\rm )} triangulation of $P$ 
and $\nabla$ a rational triangulation of $Q$. 
Let $\eta\colon P \to Q$ be piecewise linear map 
compatible with $\Delta$.   
If $\eta$ is such that for each $S\in\Delta$
there exists $T\in\nabla$ with $\eta(S)\subseteq T$, 
then there exists $\eta'\colon P\to Q$ satisfying the following conditions:
\begin{itemize}
  \item[(i)]  $\eta'$ is a piecewise linear map 
     compatible with $\Delta$,
  \item[(ii)] for each $v\in\ver(\Delta)$,\ 
     $\eta'(v)$ is a rational point, and
\item[(iii)] for each $v\in\ver(\Delta)$,\
      if $\eta(v)$ is a rational point, then $\eta(v)=\eta'(v)$.
\end{itemize}
\end{lemma}

\begin{proof}
Let $v_1,\ldots,v_k$ be the vertices of $\Delta$ 
such that  $\eta(v_i)$ is not a rational point in~$Q$. 
For each $i=1,\ldots,k$, 
let $T_i\in\nabla$ be the minimal simplex  
such that  $\eta(v_i)\in T_i$. 
For each $i=1,\ldots,k$ let us, arbitrarily, 
fix $w_i\in \ver(T_i)$. 
Let $\eta'\colon P\to \R^m$ 
be the unique piecewise linear map 
compatible with $\Delta$ such that 
\[
\eta'(v)=
\begin{cases}
  w_i&\mbox{if }v=v_i \mbox{ for some }i=1,\ldots,k,\\
  \eta(v)& \mbox{otherwise.}
\end{cases}
\]
We claim that $\eta'(P)\subseteq Q$. 
Indeed, let $S\in\Delta$. 
If  $\eta(v)=\eta'(v)$ for each $v\in\ver(S)$ 
(that is, $\{v_1,\ldots,v_k\}\cap S=\emptyset$), 
then $\eta'(S)=\eta(S)$.
Now assume that $\{v_1,\ldots,v_k\}\cap S\neq\emptyset$. 
By hypothesis, there exists $T\in\nabla$ such that 
$\eta(S)\subseteq T$.
 For each  $v_i\in\{v_1,\ldots,v_k\}\cap S$, 
by the minimality of  $T_i$  
we have the inclusion $T_i\subseteq T$, 
whence $\eta'(v_i)\in T$. 
Therefore, $\eta'(S)=\conv(\eta'(\ver(S)))\subseteq T\subseteq Q$.
\end{proof}

\subsection{Regular and strongly regular triangulations}
\label{SubSec:RegularTriang}

Recall that for any rational point $v$ in $\Q^{n}$ 
we let $\den(v)$ denote the lowest common denominator 
of the coordinates of $v$.
The vector $\widetilde{v}=\den(v)(v,1)\in\Zed^{n+1}$ 
is called the {\it homogeneous correspondent} of~$v$.
A simplex $S\subseteq \R^{n}$ is called {\it regular} 
if the set of homogeneous correspondents of its vertices 
is part of a basis of the free abelian group $\Zed^{n+1}.$
By a {\it regular triangulation} of a rational polyhedron $P$ 
we understand a rational triangulation of $P$ 
consisting of regular simplexes.

\begin{lemma}\label{lem:RegularFromRational}
Let $P\subseteq Q\subseteq \R^n$ be polyhedra 
and $\Delta$ a triangulation of~$Q$ 
such that  the simplicial complex 
$\Delta_P=\{S\in\Delta\mid S\subseteq P\}$ 
is a  triangulation of $P$.
If~$P$ is a rational polyhedron and $\Delta_P$ 
is a rational triangulation of $P$,
then there exists  a stellar subdivision $\nabla$ of $\Delta$, 
such that the simplicial complex 
$\nabla_P=\{S\in\nabla\mid S\subseteq P\}$ 
is a regular triangulation of $P$.
\end{lemma}
\begin{proof}
The affine version of \cite[Theorem 8.5]{Ewa1996}  
yields a regular triangulation $\Sigma$ of $P$ 
that is  a stellar subdivision  of~$\Delta_P$. 
\end{proof}

A simplex $S$ is said to  be  {\it strongly regular} 
if it is regular and 
the greatest common divisor of the denominators of
the vertices of $S$ is equal to $1$.
A  rational triangulation $\Delta$  of a rational polyhedron $P$ is
said to be {\it strongly regular}  if each maximal simplex of $\Delta$  
is strongly regular.

\begin{lemma}[{\cite[Lemma~3.3]{CM2013}}]
\label{Lem:StronglyRegStable}
Let  $\Delta$ and $\nabla$ be regular triangulations
of a rational polyhedron $P\subseteq\R^n$.
Then $\Delta$ is strongly regular if and only if  $\nabla$ is.
\end{lemma}

\begin{lemma}\label{Lem:SRandConvexSet}
Let $P\subseteq\R^{n}$ be a rational polyhedron. 
Then the following  conditions  are equivalent:
\begin{itemize}
  \item[(i)] there exists a strongly regular triangulation of $P$;
  \item[(ii)] for each rational point $v\in P$ 
 there exist $w\in \Zed^n$ and $\varepsilon>0$
 such that the convex segment $\conv(v,v+\varepsilon (w-v))$
 is contained in~$P$.
\end{itemize}
\end{lemma}
\begin{proof}
(i)$\Rightarrow$(i). Let $v\in P\cap\Q^n$.
 If $v\in\Zed^n$, the claim follows trivially by setting $w=v$  since $\conv(v,v+\varepsilon(w-v))=\{v\}\subseteq P$ for each $\varepsilon>0$. 
So assume that $v\notin\Zed^n$. Let $\Delta$ be a strongly regular triangulation of $P$ and $S$ be a strongly regular simplex in $\Delta$ such that $v\in S$. By \cite[Theorem~4.4]{Cab201X}, there exists $u\in S$ such that $\gcd(\den(v),\den(w))=1$. Let $a,b\in\Zed$ be such that $a>0$ and $a\den(v)+b\den(w)=1$. Since $v\notin \Zed^n$, \ $\den(v)\neq 1$ and $b\neq 1$. Let us set $w=a\den(v) v+b\den(u)u\Zed^n$ and $\varepsilon=1/(b\den(u))$. Then
$v+ \varepsilon (w-v)=u$ and $\conv(v,v+ \varepsilon (w-v))\subseteq S\subseteq P$, which proves that $P$ satisfies condition (ii).

(ii)$\Rightarrow$(i) This is just an application of \cite[Lemma~4.1]{Je2010} and \cite[Theorem~4.17]{Cab201X}.
\end{proof}
\begin{lemma}\label{Lem:ExistsCoprime}
Let $S$ be a strongly regular simplex. 
For each $k=1,2,\ldots$ there exists a rational point $v\in S$
such that $\gcd(k,\den(v))=1$.
\end{lemma}

\begin{proof}
By \cite[Lemma~3.4]{CM2013}, for a suitably large $m\in\{1,2,\ldots\}$, 
there exists a rational point $v\in S$ 
such that $\den(v)$ divides $(k\cdot m)+1$. 
Therefore, $\den(v)$ and $k$ are coprime.
\end{proof}

\subsection{$\Zed$-maps and weighted abstract simplicial complexes}
Take $P\subseteq \R^{n}$ and $Q\subseteq\R^m$ rational polyhedra, 
and  $\eta\colon P\rightarrow Q$ a piecewise linear map.
If there is a  triangulation $\Delta$  of $P$
such that the restriction of $\eta$ to every simplex $T$  
of  $\Delta$ coincides on $T$
with an (affine) linear map 
$\eta_{T}\colon \R^{n}\rightarrow \R^{m}$ with integer coefficients, 
then $\eta$ is called a {\it $\Zed$-map}. 
In other words, $\eta$ is a $\Zed$-map if and only if 
it is a piecewise linear map 
and each linear piece has integer coefficients. 
Since $\Zed$-maps are piecewise linear, 
if a $\Zed$-map $\eta$ is compatible
 with a rational triangulation $\Delta$, 
then $\eta$ is uniquely determined by its values on $\ver(\Delta)$.

\begin{lemma}[{\cite[Lemma~3.7]{Mu2011}}]\label{Lem-LinearMap}
Let $P\subseteq \R^{m}$  and $Q\in\R^n$ be rational polyhedra 
and $\Delta$ a regular triangulation of $P$.
If $\eta\colon P\to Q$ is a piecewise linear map 
compatible with $\Delta$, 
then the following conditions are equivalent:
\begin{itemize}
\item[(i)] $\eta$ is a $\Zed$-map;
\item[(ii)] for each $v\in \ver(\Delta)$, 
	$\eta(v)\in\Q^n$ and $\den(\eta(v))$ divides $\den{v}$.
\end{itemize}
\end{lemma}

A {\it weighted abstract simplicial complex} 
is a triple $(V,\Sigma, \omega)$
where $(V,\Sigma)$ is an (always finite) abstract simplicial complex
and $\omega$ is a map from $V$ into the set $\{1,2,3,\ldots\}.$

Let $\mathfrak{W}=(V,\Sigma, \omega)$ 
be a weighted abstract simplicial complex 
with vertex set $V=\{v_{1},\ldots,v_{n}\}$.
Let $e_{1},\ldots,e_{n}$ be the standard basis vectors 
of $\R^{n}$.
We then use the notation $\Delta_{\mathfrak{W}}$
for the complex whose vertices  are 
\[
   v'_{1} = e_{1}/\omega(v_{1}),\ldots,v'_{n}=e_{n}/\omega(v_{n})\in\R^n,
\]
and whose $k$-simplexes ($k=0,\ldots,n$) are given by
\[
  \conv(v'_{i(0)},\ldots, v'_{i(k)})\in \Delta_{\mathfrak{W}}\quad
      \text{ if and only if }\quad
   \{ v_{i(0)},\ldots, v_{i(k)}\}\in \Sigma.
\]
Trivially, $\Delta_{\mathfrak{W}}$ is a regular triangulation 
of the polyhedron $|\Delta_{\mathfrak{W}}|\subseteq [0,1]^{n}$.  
The polyhedron $|\Delta_{\mathfrak{W}}|$ 
is called the {\it geometric realisation} of $\mathfrak{W}$.

\section{$\Zed$-retracts and finitely generated projective unital $\ell$-groups}

\begin{definition}
A rational polyhedron $P\subseteq [0,1]^n$ is said to be a 
{\it $\Zed$-retract} 
if there exists a $\Zed$-map $\eta\colon [0,1]^n\to P$ 
such that 
\[
\eta(v)=v\mbox{ for each }v\in P.
\]
\end{definition}

The following result shows that
 the notion of $\Zed$-retract is independent of the cube $\cube$.

\begin{lemma}[{\cite[Lemma~4.2]{CM2013}}]
\label{Lem:Independence}
Let $P\subseteq\cube$ be a rational polyhedron. 
If there is $m\in\{1,2,\ldots\}$ 
and $\Zed$-maps $\mu\colon [0,1]^m\to P$ 
and $\nu\colon P\to [0,1]^m$ 
such that $\mu(\nu(v))=v$ for each $v\in P$, 
then $P$ is a $\Zed$-retract.
\end{lemma}

By definition, $\McN(P)$ is the set of all 
real-valued $\Zed$-maps defined on $P$.
 The connection between $\Zed$-retracts 
and finitely generated unital $\ell$-groups
 is given by the following:

\begin{theorem}[{\cite[Theorem~4.1]{CM2013}}]
\label{Thm_Projective_Retraction} 
A  unital $\ell$-group $(G,u)$ is finitely generated projective 
if and only if it is isomorphic to $\McN(P)$ for some $\Zed$-retract $P$.
\end{theorem}

In Theorem~\ref{Thm:Main}, the main result of this paper, 
we present necessary and sufficient conditions 
for a rational polyhedron to be a $\Zed$-retract. 
The proof relies on the following three results:

\begin{theorem}[{Whitehead's Theorem--contractible case \cite[Theorem~4.5]{Hat2001}}]
\label{Thm:Whitehead}
Let $P\subseteq [0,1]^n$ be a polyhedron. 
If $P$ is contractible, then $P$ is a {\em deformation retract} of $[0,1]^n$, 
that is, there exists a retraction $f\colon [0,1]^n\to P$ 
homotopically equivalent to the identity on $[0,1]^n$ 
relative to $P$.  
\end{theorem}

\begin{theorem}[{Relative  Simplicial Approximation Theorem 
\cite[Theorem~3.2.1]{Sta1967}}]
\label{Thm:SimplicialApproximation}
Let $P\subseteq Q\subseteq \R^n$ and $R\subseteq \R^m$ 
be polyhedra 
and $f\colon Q\to R$ be a continuous map 
such that $f{\upharpoonright}_P$ is a piecewise linear. 
Then there exists a piecewise linear map  $g\colon Q\to R$  
 homotopically equivalent to $f$  such that 
$f$  and agrees with   $g$ on $P$ 
(in symbols,   $g{\upharpoonright}_P = f{\upharpoonright}_P$).
\end{theorem}

\begin{theorem}[{\cite[Theorem~5.1]{CM2013}}]
    \label{theorem:collapsible}
Let $P\subseteq [0,1]^{n}$ be a   polyhedron. 
Assume
\begin{itemize}
\item[(i)] $P$ has a collapsible  triangulation $\nabla$,
\item[(ii)] $P$ contains a vertex $v$ of $\cube$, and
\item[(iii)] $P$ has a strongly regular triangulation  $\Delta$.
\end{itemize}
Then  $P$ is a $\Zed$-retract.
\end{theorem}

We are now ready to prove the main result of this paper.

\begin{theorem}\label{Thm:Main}
Let $P\subseteq \cube$ be a rational  polyhedron. 
Then the following conditions are equivalent:
\begin{itemize}

  \item[(a)] $P$ is a $\Zed$-retract.
  
  \item[(b)] $P$ has the following properties:
    \begin{itemize}
        \item[(i)] $P$ is contractible,
        \item[(ii)] $P\cap\{0,1\}^n\neq\emptyset$, and
        \item[(iii)] $P$ has a strongly regular triangulation.
     \end{itemize}

\end{itemize}

\end{theorem}

\begin{proof}
Theorems~\ref{Thm:ToInvert} and \ref{Thm_Projective_Retraction} 
take care of  (a)$\Rightarrow$(b).
\medskip

(b)$\Rightarrow$(a). 
Assume that $P$ satisfies (i)--(iii).
We divide the proof into two parts.
The first part is devoted to contructing a map $\eta\colon\cube\to P$ 
and a  (not necessarily rational) triangulation $\Delta$ of $\cube$ 
satisfying the following conditions:
\begin{itemize}
\item[(a)] $\eta$ is a continuous retract  onto $P$;
\item[(b)] $\eta$ is piecewise linear;
\item[(c)] $\Delta$ is collapsible;
\item[(d)] $\eta$ and $\Delta$ are compatible;
\item[(e)] the simplicial complex $\Delta_P=\{S\in\Delta\mid S\subseteq P\}$ is a triangulation of $P$;
\item[(f)]  $\Delta_P$ is rational and regular;
\item[(g)] $\eta(v)$ is a rational point for each $v\in\ver(\Delta)$;
\item[(h)] for each $n$-simplex $\conv( v_0,\ldots,v_n )\in\Delta$,\  \[\gcd(\den(\eta(v_0),\ldots,\den(\eta(v_n)))=1.\] 
\end{itemize} 
In the second part, using $\eta$ and $\Delta$, 
we will prove that $P$ is a $\Zed$-retract.
\medskip

\noindent{\bf Part 1:}
The map $\eta$ and the triangulation $\Delta$, satisfying (a)--(h) above, are obtained 
in a series of steps.
Each step provides a  map and/or a triangulation 
that increasingly satisfy the conditions listed above. 
We label the steps with the capitalised letter 
of the label of the property that the new map and/or 
the triangulation obtained will have--in addition to the ones 
already satisfied by the previous maps and triangulations. 
\medskip

\noindent{\it Step} (A): 
By Theorem~\ref{Thm:Whitehead}, 
there exists a continuous retraction $\eta_a\colon[0,1]^n\to P$. 

\medskip

\noindent{\it Step} (B):
An application of the Simplicial Approximation Theorem 
(Theorem~\ref{Thm:SimplicialApproximation}) 
provides a continuous piecewise linear map 
$\eta_b\colon[0,1]^n\to P$ such that 
$\eta_b{\upharpoonright}_{P}=\eta_a{\upharpoonright}_{P}$. 
Therefore, $\eta_b$ is also a continuous retraction
 of $[0,1]^n$ onto $P$.
\medskip

\noindent{\it Step} (C):
Let $\Delta_c$  be the  
standard triangulation of  $\cube$  
in Example~\ref{Ex:CollapsibleCube}.
(Any other collapsible triangulation of $\cube$ would do.)
\medskip

\noindent{\it Step} (D): 
Let $\Sigma$ be a triangulation of $[0,1]^n$ such that 
$\eta_b{\upharpoonright}_{S}$ is linear for each $S\in\Sigma$.
Lemma~\ref{Lem:BlowUpSubdivision} yields a stellar subdivision  $\Delta_d$ of $\Delta_c$ 
such that $\Delta_d$ is also a subdivision of $\Sigma$. 
Therefore, $\eta_b$ is linear on each simplex of $\Delta_d$, 
and by Lemma~\ref{Lem:BlowUpsPreserveCollapsible}, 
$\Delta_d$ is collapsible.
\medskip

\noindent{\it Step} (E):
Since $P\subseteq\cube$ is a rational polyhedron, 
Corollary~\ref{Cor:Triang-Subset1} yields a stellar subdivision 
$\Delta_e$ of  $\Delta_d$ such that the simplicial complex 
$(\Delta_e)_P=\{S\in\Delta_e\mid S\subseteq P\}$ 
is a triangulation of $P$. 
Clearly  $\eta_b$ and $\Delta_e$ are compatible, 
and by Lemma~\ref{Lem:BlowUpsPreserveCollapsible},
 $\Delta_e$ is collapsible since $\Delta_d$ is.
\medskip

\noindent{\it Step} (F):
An application of Corollary~\ref{cor:Beynon} to  $\Delta_e$ 
provides a triangulation $\Delta_e'$  of $\cube$ 
that is simplicially isomorphic to $\Delta_e$ 
and such that the simplicial complex 
$(\Delta_e')_P=\{S\in\Delta_e'\mid S\subseteq P\}$ 
is a rational triangulation of $P$. 
Now Lemma~\ref{lem:RegularFromRational} 
gives  a stellar subdivision $\Delta_f$  of $\Delta_e'$ 
such that the simplicial complex 
$(\Delta_f)_P=\{S\in\Delta_f\mid S\subseteq P\}$ 
is a regular triangulation of~$P$. 
Since $\Delta_e$ is collapsible 
and  $\Delta_e'$ is simplicially isomorphic to $\Delta_e$, 
the triangulation  $\Delta_e'$ of $\cube$ is collapsible, 
whence so is its stellar subdivision $\Delta_f$.
\medskip

\noindent{\it Step} (G):
Let $\eta_f\colon\cube\to P$ be the unique piecewise linear map 
compatible with $\Delta_f$ such that 
\begin{equation}\label{Eq:Etaf}
  \eta_f(v)=
    \begin{cases}
	 v&\mbox{if }v\in \ver(\Delta_f)\cap P,\\
	 \eta_b(v) & \mbox{if }v\in \ver(\Delta_f)\setminus P.
    \end{cases}
\end{equation}
Corollary~\ref{Cor:RefineByBlowUps} 
yields a stellar subdivision $\Delta_g$ of 
$\Delta_f$ satisfying the following conditions:
\begin{itemize}
\item[(G-i)]  for each ${S\in\Delta_g}$ there is
$T\in(\Delta_f)_P$  such that $\eta_f(S)\subseteq T$;
\item[(G-ii)] if $S\in\Delta_f$ 
is such that there exists 
$T\in(\Delta_f)_{P}$ with $\eta(S)\subseteq T$, 
then $S\in\Delta_g$.
\end{itemize}
Let $S\in (\Delta_f)_P$. Since $\eta_f(S)=S$,
 from (G-ii) it follows that 
$S\in(\Delta_g)_P=\{T\in\Delta_g\mid T\subseteq P\}$. 
Therefore,
\begin{equation}\label{Eq:DeltafDeltag}
   (\Delta_f)_P=(\Delta_g)_P.
\end{equation}
Since $\Delta_f$ is a stellar subdivision of $\Delta_g$, 
the map $\eta_f$ is compatible with $\Delta_g$. 
An application of Lemma~\ref{Lem:RationalizationOfMap}
provides a map $\eta_g\colon\cube\to P$ such that
\begin{itemize}
  \item[(G-iii)] for each $v\in\ver(\Delta_g)$
    the point $\eta_g(v)$ is  rational, and
  \item[(G-iv)] for each $v\in\ver(\Delta_g)$, 
    if $\eta_f(v)$ is a rational point then $\eta_g(v)=\eta_f(v)$.
\end{itemize}
Since $(\Delta_f)_P$ is a rational triangulation of $P$, 
each vertex of $(\Delta_f)_P$ is a rational point. 
Thus from condition (G-iv) and \eqref{Eq:Etaf} it follows that
\begin{equation}\label{Eq:Retract}
  \eta_g(v)=\eta_f(v)=v\,\,\mbox{ for each }v\in P.    
\end{equation}
By Lemma~\ref{Lem:BlowUpsPreserveCollapsible}, 
$\Delta_g$ is collapsible since $\Delta_f$ is.
\medskip

\noindent{\it Step} (H):
Since $\Delta_g$ is a triangulation of $[0,1]^n$,
 a simplex in $\Delta_g$ is maximal if and only if it is an $n$-simplex.
Let $S_1,\ldots,S_k$ be the list of all $n$-simplexes of 
$\Delta_g$ 
 $\gcd(\{\den(\eta(v))\mid v\in\ver(S)\})\neq 1$.  
For each $i=1,\ldots k$, 
let $v_i$ be an arbitrary point in the relative interior of $S_i$, 
that is, $v_i$ belongs to $S_i$ 
but  does not belong to any proper face of $S_i$ 
(for instance, we may let each $v_i$ be the barycentre of~$S_i$). 
Set $\nabla_0=\Delta_g$ and $\nabla_i=(\nabla_{i-1})_{v_i}$ 
for each $i=1,\ldots,k$. 
Let $\Delta_h=\nabla_k$. 
Since $\Delta_g$ is collapsible and $\Delta_h$ is a stellar 
subdivision of $\Delta_g$, 
it follows from Lemma~\ref{Lem:BlowUpsPreserveCollapsible} that 
$\Delta_h$ is collapsible.
 By (iii), $P$ admits a strongly regular triangulation.
Since $(\Delta_g)_P$ is a regular, 
Lemma~\ref{Lem:StronglyRegStable} shows that 
$(\Delta_g)_P$ is strongly regular. 
Therefore, $\{S_1,\ldots,S_k\}\cap(\Delta_g)_P=\emptyset$, 
and hence
\begin{equation}\label{Eq:DeltahDeltag}
  (\Delta_h)_P=(\Delta_g)_P.
\end{equation}
For each $i=1,\ldots,k$, 
let $T_i\in(\Delta_g)_P=(\Delta_h)_P$ be such that 
$\eta_g(S_i)\subseteq T_i$ and $T_i$ is maximal in $(\Delta_h)_P$. 
By Lemma~\ref{Lem:ExistsCoprime}, 
for each $i=1,\ldots,k$, there exists a rational point $x_i\in T_i$ such that $\den(x_i)$ and $\den(\eta(v))$ are coprime
for each $v\in\ver(S_i)$.
Let $\eta_h$ be the unique piecewise linear  map compatible
 with $\Delta_h$ such that 
\begin{equation}\label{Eq:DefEta}
  \eta_h(v)=
\begin{cases}
   \eta_g(v)&\mbox{if } v\in\ver(\Delta_g), \\ 
   x_i&\mbox{if }v=v_i\mbox{ for some } i=1,\ldots,k. 
\end{cases}
\end{equation}
Then $\eta_g$ satisfies (h).
\bigskip

To conclude the first part of the proof let us set 
 $\eta=\eta_h$ and $\Delta=\Delta_h$. 
From Step~(H),
 it follows directly that $\eta$ and $\Delta_h$
 satisfy conditions (b)-(d) and (h). 
Since $(\Delta_f)_P$ is a regular triangulation of $P$, 
from \eqref{Eq:DeltafDeltag} and \eqref{Eq:DeltahDeltag} 
it follows that $\Delta$ satisfies (e) and (f). 
From {(G-i)} and \eqref{Eq:DefEta} it follows that $\eta$ satisfies (g). 
By definition, $\eta$ is continuous. 
Equations \eqref{Eq:Retract} and \eqref{Eq:DefEta} show that 
$\eta(v)=v$ for each $v\in\ver(\Delta_P)$. 
Since $\Delta_P$ is a triangulation of~$P$ 
and $\eta$ is compatible with $\Delta$, 
we obtain that  $\eta(v)=v$ for each $v\in P$, 
thus proving that $\eta$ satisfies (a). 

Therefore, we have proved the existence of 
$\eta$ and $\Delta$ satisfying conditions (a)--(h). 

\bigskip

\noindent{\bf Part 2:} 
Let the abstract simplicial complex $(\ver(\Delta),\Sigma_{\Delta})$ 
be the skeleton of~$\Delta$ (see Section~\ref{Subsec:AbsSimCom}). 
Let  $\mathfrak{W}=(\ver(\Delta),\Sigma_{\Delta},\omega)$  
be the weighted abstract simplicial complex such that   
\begin{equation}\label{Eq:omega}
  \omega(v)=\den(\eta(v))\quad \mbox{ for each } v\in\ver(\Delta).
\end{equation}
Let $v_1,\ldots,v_k$ be the vertices of $\Delta$,
and  $Q=|\Delta_{\mathfrak{W}}|\subseteq [0,1]^k$ 
the geometric  realisation of $\mathfrak{W}$. 
By (ii), there exists $v\in P$ such that $\den(v)=1$. 
Since $\Delta_P$ is a regular triangulation of $P$,\, 
$v=v_i\in\ver(\Delta_P)\subseteq\ver(\Delta)$ for some 
$i\in\{1,\ldots,k\}$ and $\omega(v_i)=\den(\eta(v))=\den(v)=1$. 
Thus $e_i\in Q\cap \{0,1\}^k$. 
Since $\Delta$ is a triangulation of $[0,1]^n$, 
a simplex $S$ is maximal in $\Delta$ if and only if it is an $n$-simplex. 
As a consequence of~(h), for each $S$ maximal in $\Delta$,
we have $\gcd(\{\omega(v)\mid v\in\ver(S)\})=1$.
Thus $\Delta_{\mathfrak{W}}$ 
is a strongly regular triangulation of $Q$. 
Hence, Theorem~\ref{theorem:collapsible} 
implies that $Q\subseteq [0,1]^k$ is a $\Zed$-retract. 
Let $\nu\colon [0,1]^k\to Q$  be a $\Zed$-map 
such that $\nu(v)=v$ for each $v\in Q$. 

Now, 
let  $\mu\colon Q\to P$ be the unique piecewise linear map 
compatible with $\Delta_{\mathfrak{W}}$ such that
\begin{equation}\label{Eq:mu}
\mu(e_i/\omega(v_i))=\eta(v_i)\mbox{ for each }v_i\in\ver(\Delta).
\end{equation} 
By \eqref{Eq:omega} and Lemma~\ref{Lem-LinearMap}, 
\ $\mu$ is a $\Zed$-map.
Similarly, 
let $\xi\colon P\to Q$  be the unique piecewise linear map 
compatible with $\Delta_P$ that satisfies
\begin{equation}\label{Eq:xi}
   \xi(v_i)=e_i/\omega(v_i)\mbox{ for each }   
v_i\in\ver(\Delta_P)=\ver(\Delta)\cap P.
\end{equation} 
By \eqref{Eq:omega}, 
another application of Lemma~\ref{Lem-LinearMap}
shows that $\xi$ is a $\Zed$-map. 

Clearly, $\mu\circ\xi$ is compatible with $\Delta_P$.
From \eqref{Eq:mu} and \eqref{Eq:xi}, it follows
$\mu\circ\xi(v_i)=\mu(e_i/\omega(v_i))=\eta(v_i)=v_i$ 
for each  $v_i\in\ver(\Delta_P)$. 
Therefore, $(\mu\circ\xi)(v)=v$ for each $v\in P$.

In conclusion,  the $\Zed$-maps $\mu\circ\nu\colon[0,1]^k\to P$ 
and $\xi\colon P\to Q\subseteq [0,1]^k$ 
satisfy $(\mu\circ\nu)\circ\xi(v)=v$ for each $v\in P$. 
By Lemma~\ref{Lem:Independence},
$P$ is a $\Zed$-retract, and the proof is complete.
\end{proof}

Combining  Lemma~\ref{Lem:SRandConvexSet} with
Theorems~\ref{Thm_Projective_Retraction}, and \ref{Thm:Main} 
we obtain the following:
\begin{theorem}\label{cor:last}
A unital $\ell$-group  $(G,u)$ is finitely generated and projective 
if and only if
there exist $n\in\{1,2,,\ldots\}$ and a rational  polyhedron
$P\subseteq \cube$ such that $(G,u)\cong\McN(P)$ and
\begin{itemize}
\item[(i)] $P$ is contractible,

\item[(ii)] $P\cap\{0,1\}^n\neq\emptyset$, and

\item[(iii)] 
for each rational point $v\in P$ 
 there exist $w\in \Zed^n$ and $\varepsilon>0$ such that the convex segment $\conv(v,v+\varepsilon (w-v))$ is contained in~$P$.
\end{itemize}
\end{theorem}

If $S\subseteq \R^n$ is an $n$-simplex, 
then $S$ automatically satisfies condition (iii) in Corollary~\ref{cor:last}.

\begin{corollary}
Let $P$ be a rational  polyhedron $P\subseteq \cube$ such that 
\begin{itemize}
\item[(i)] $P$ is contractible,

\item[(ii)] $P\cap\{0,1\}^n\neq\emptyset$, and

\item[(iii)] $P=S_1,\cup\cdots\cup S_k$ for suitable $n$-simplexes $S_1,\ldots, S_k\subseteq \cube$.
\end{itemize}
Then $\McN(P)$ 
is a finitely generated projective unital $\ell$-group.
\end{corollary}

\subsection*{Acknowledgements}
I would like to thank Daniele Mundici  for his numerous comments and corrections. 
I also wish to thank Hilary Priestley for her careful reading of a previous draft of this paper and her valuable suggestions for improvement.

\end{document}